\documentclass[11pt]{amsart}
\usepackage[margin=1in]{geometry}
\usepackage{xcolor}
 \usepackage[english]{babel}
\numberwithin{equation}{section}
\newtheorem{thm}{Theorem}[section]
\newtheorem{prop}[thm]{Proposition}
\newtheorem{lem}[thm]{Lemma}

\newtheorem{conj}[thm]{Conjecture}

\newtheorem{clm}[thm]{Claim}

\def\Xint#1{\mathchoice
	{\XXint\displaystyle\textstyle{#1}}%
	{\XXint\textstyle\scriptstyle{#1}}%
	{\XXint\scriptstyle\scriptscriptstyle{#1}}%
	{\XXint\scriptscriptstyle\scriptscriptstyle{#1}}%
	\!\int}
\def\XXint#1#2#3{{\setbox0=\hbox{$#1{#2#3}{\int}$ }
		\vcenter{\hbox{$#2#3$ }}\kern-.6\wd0}}

\def\dashint{\Xint-}

\newcommand{\lp}{M_{x_0,r}}
\newcommand{\smp}{m_{x_0,r}}
\newcommand{\op}{\omega_{x_0,r}}
\newcommand{\mnp}{(\m \cdot\nabla p) }

\newcommand{\mnw}{(\m \cdot\nabla w) }
\newcommand{\divg}{\textup{div}}

\newcommand{\ot}{\Omega_T}

\newcommand{\rn}{\mathbb{R}^N}
\newcommand{\rt}{\mathbb{R}^2}

\newcommand{\ob}{\partial\Omega}
\newcommand{\pt}{\partial_t}

\newcommand{\br}{B_{r}(x_0)}

\newcommand{\brn}{B_{r_n}(x_0)}

\newcommand{\bRy}{B_{r}(y)}

\newcommand{\brh}{B_{\frac{r}{2}}(x_0)}
\newcommand{\ibr}{\int_{B_{r}(x_0)}}

\newcommand{\ibnx}{\int_{B_{r_n}(x_0)}}

\newcommand{\ibrxj}{\int_{B_{\frac{r}{2^{j-1}}}(x)}}

\newcommand{\ibRy}{\int_{B_{r}(y)}}
\newcommand{\ibryj}{\int_{B_{\frac{r}{2^{j-1}}}(y)}}

\newcommand{\aibr}{\dashint_{B_{r}(x_0)}}

\newcommand{\ra}{\rightarrow}
\newcommand{\io}{\int_{\Omega}}

\newcommand{\ve}{\varepsilon}

\newcommand{\m}{\mathbf{m}}

\newcommand{\kbr}{\overline{k}(r)}
\newcommand{\ece}{(I+\m\otimes\m)}

\begin{document}
	\title[Modulus of continuity of weak solutions]{%A partial result to a conjecture of De Giorgi and its applications an elliptic-parabolic system modeling biological transportation networks and a biological network formation model
	Modulus of continuity of weak solutions to a class of singular elliptic equations }
	\author{Xiangsheng Xu}\thanks
	%\addressNon-uniqueness of weak solutions to a doubly nonlinear fourth order elliptic equation  H\"{o}lder continuity of weak solutions to an elliptic-parabolic system modelingand its applications to a biological network formation model
	{Department of Mathematics and Statistics, Mississippi State
		University, Mississippi State, MS 39762.
		{\it Email}: xxu@math.msstate.edu. {\it Ann. Scuola Norm. Sup. Pisa Cl. Sci.}, to appear.}
	\begin{abstract} In this paper we study the modulus of continuity of weak solutions to a singular elliptic equation in the plane under very weak assumption on the integrability of the elliptic coefficients. Our investigation reveals that the modulus of continuity can be described by the reciprocal of the logarithmic function raised to a power. However, the power can be arbitrarily large. This is in sharp contrast with a result by J. Onninen and X. Zhong ({\it Ann. Scuola Norm. Sup. Pisa Cl. Sci.}, {\bf 6}(2007), 103–116) for a degenerate elliptic equation in the plane, in which the power must be suitably small.
	\end{abstract}
	\maketitle

\section{Introduction}
In this paper we investigate the modulus of continuity of a weak solution to the Dirichlet boundary value problem
%s to a class of second-order, singular elliptic equations of the form
\begin{eqnarray}\label{ell0}
	-\mbox{div}\left[\ece\nabla p\right]&=&S\ \ \mbox{in $\Omega$,}\\
	p&=&0\ \ \mbox{on $\ob$}\label{el1}
\end{eqnarray}
under the following assumptions:
\begin{enumerate}
	\item[(H1)]$\Omega$ is a bounded domain in $\mathbb{R}^2$  with Lipschitz boundary $\ob$;
		\item[(H2)]$S\in L^q(\Omega)$ for some $q>1$;
		\item[(H3)]$\m=(m_1,m_2)$ is a  vector-valued function in $\left(L^4(\Omega)\right)^2$.
\end{enumerate} 
% the matrix whose $ij$-entry is $m_im_j$
Recall that the out product $\m\otimes \m$ is the matrix given by
$$\m\otimes \m=\m^T \m.$$
Thus,
\begin{equation}
	\m\otimes \m \nabla p=\mnp\m.\nonumber
\end{equation}
We say that $p$ is a weak solution to \eqref{ell0}-\eqref{el1} if:
\begin{enumerate}
	\item[(D1)]  $p\in W_0^{1,2}(\Omega),\ \m\cdot\nabla p\in L^2(\Omega)$;
	\item[(D2)] for each $\zeta \in W_0^{1,2}(\Omega)$ with $\m\cdot\nabla \zeta\in L^2(\Omega)$ one has
	\begin{equation}\label{el3}
		\io\left[\nabla p\nabla\zeta+(\m\cdot\nabla p)(\m\cdot\nabla \zeta)\right]dx=\io S(x)\zeta dx.
	\end{equation}
\end{enumerate}
Note from Theorem 7.15 in \cite{GT} that the test function $\zeta$ in (D2) satisfies
\begin{equation}
	\io e^{c_0|\zeta|^2}dx<\infty\ \ \mbox{for some positive number $c_0$}.\nonumber
\end{equation}
Thus, each integral in \eqref{el3} is well-defined.

The out product of two vectors appears in many mathematical models. Our situation here is directly related to the analysis of biological transport networks \cite{HMP,H,HC}. We summarize known results concerning \eqref{ell0}-\eqref{el1} in the following
\begin{prop}Let (H1)-(H3) Hold. Then there is a unique weak solution $p$ to \eqref{ell0}-\eqref{el1}. Moreover,
	\begin{enumerate}
		\item[\textup(C1)] there is a positive number $c=c(\Omega)$ such that
		$$\|	p\|_{\infty,\Omega}\leq c\|S\|_{q,\Omega},\ \ \mbox{and}$$
		\item[\textup(C2)] for each $x_0\in\Omega$ and each  $R\in \left(0,\min\{1, \textup{dist}(x_0,\ob)\}\right)$ we can find a positive constant $c$ with the property
		\begin{equation}%\label{hc1}
			\underset{\br}{\textup{osc}  }\ p\equiv\underset{\br}{\textup{ess sup}}\ p-\underset{\br}{\textup{ess inf}}\ p
			\leq \frac{c}{\ln^{\frac{1}{2}}\frac{R}{r}}\ \ \mbox{for all $r\in (0, R]$},\nonumber
		\end{equation}
	where $\br$ is the ball centered at $x_0$ with radius $r$.
		\end{enumerate}
	\end{prop}
The existence of a unique weak solution can be inferred from a result in \cite{X3}. The conclusion (C1) is true even for the space dimension $N>2$ \cite{LX}. A suitable modification of the proof in \cite{LX} can be applied to the case $N=2$ (see Claim \ref{cpb} below). The continuity property (C2) was established in \cite{X5}.

To describe our results here, we introduce some notations first. 
We say that $f\in K_2(\Omega)$, the Stummel-Kato class of functions \cite{CFG,K}, if  $f$ is a measurable function on  $\Omega$ and%. For each $r>0$ define
\begin{equation}\label{kato1}
	\eta(f;\Omega;r)\equiv\sup_{y\in \Omega}\ibRy|f(x)|\chi_{\Omega}\left|\ln|x-y|\right| dx\ra0\ \ \mbox{ as $r\ra 0^+$},
\end{equation}
%$$$$If
%Let $\Omega_0$ is an open subset of $\mathbb{R}^2$.
%, then  $\eta(f;r; \Omega_0)$ is defined to be $\eta(f\chi_{\Omega_0};r)$. 
% if $\lim_{r\ra 0}\eta(f\chi_{\Omega_0};r)=0$ 
while
  $f\in K_2^{\textup{loc}}(\Omega)$ means that $\lim_{r\ra 0}\eta(f;\Omega_1;r)=0$ for each bounded subdomain $\Omega_1$ of $\Omega$ with $\overline{\Omega_1}\subset\Omega$. Note that in \eqref{kato1} we have used $y\in\Omega$ instead of $y\in \rt$ as was done in \cite{K}. Our definition here seems to be more suitable for PDE applications.
% Obviously, we have
%$$  K_2^{\textup{loc}}(\Omega_0)\subset \mathcal{H}_{\mbox{loc}}(\Omega_{0}).$$
%f\in $, then $\eta(f;r)$ is finite for each $r>0$. 
%In view of \cite{CFG,K}, to establish \eqref{main3}, we need the condition
%\begin{equation}
%	\sup_{y\in\br}\ibr|\nabla p|^2|\ln|x-y||dx.
%\end{equation}
%for the local boundedness of $w$. 
%That is, 
%\begin{equation}
%	|\nabla p|^2\in  K_2^{\textup{loc}}(\Omega)\ \ ,
%\end{equation} . 
%According to ,  depends on
As usual, the letter $c$ or $c_i, i=0,1,\cdots$, will be used to represent a generic positive constant.

 Our main result is:
\begin{thm}\label{main} Let (H1)-(H3) hold and $p$ be the weak solution  to \eqref{ell0}-\eqref{el1}.
	\begin{enumerate}
	%	\item[\textup(C1)] $\|	p\|_{\infty,\Omega}\leq c\|S\|_{q,\Omega}$;
			\item[\textup(C3)] Then for each $y\in\Omega$, $\ell>0$ there is a positive constant $c=c\left(\ell, \textup{dist}(y,\ob)\right)$ such that
			\begin{equation}\label{hc2}
				\underset{\bRy}{\textup{osc}  }\ \ p
				\leq \frac{c}{\ln^\ell\frac{R}{r}}+c  r^{\frac{2(q-1)}{q}}\ \ \mbox{for all $r\in (0, R]$},
				\end{equation}
			where $R\in \left(0,\min\{1, \textup{dist}(y,\ob)\}\right)$;
				\item[\textup(C4)] If we further assume that 
				\begin{enumerate}
					\item[\textup{(H4)}] $\m\in \left(\textup{BMO}(\Omega)\right)^2$ \cite{DF},
				\end{enumerate} then
				\begin{equation}\label{ko5}
					F\equiv|\nabla p|^2+\mnp^2\in  K_2^{\textup{loc}}(\Omega).
				\end{equation}
				\end{enumerate}
 \end{thm}
%If $y\in\ob$, then (C3) still holds with $\bRy$ being replaced by $\bRy\cap\Omega$.

To put our results into some mathematical context, we recall
%Our problem here is related to 
a conjecture by De Giorgi \cite{C,DE}. In 1995, De Giorgi gave a lecture in Lecce, Italy on the continuity of weak solutions to second-order elliptic equations of the form
\begin{equation}\label{ell1}
	\mbox{div}(A\nabla u)=0\ \ \mbox{in $\Omega\subset\rn, \ N\geq 2$}.
\end{equation}
Here the entries of the coefficient matrix $A=A(x)$ are measurable functions, satisfying
	\begin{equation}\label{ell3}
\lambda(x)	|\xi|^2\leq A(x)\xi\cdot\xi\leq \Lambda(x)|\xi|^2\nonumber
\end{equation}
%$$$$
for some non-negative measurable functions $\lambda(x), \Lambda(x)$, a.e $x\in \Omega$, and each $\xi\in \rn$. Obviously, if $\Lambda(x)$ is bounded above, we can replace it with a positive constant, and we can do the same with $\lambda(x)$ if it is bounded away from $0$ below. With this in mind, we say that \eqref{ell1} is singular if $\lambda$ is a constant while $\Lambda$ is not, and degenerate if $\Lambda$ is a constant and $\lambda$ is not. %and 
% Obviously, we assume that $\Lambda(x)$ is not bounded above and $\lambda(x)$ is not bounded away from $0$ below.
De Giorgi proposed several open problems \cite{DE} about the equation, one of which was stated as follows: 
\begin{conj}[De Giorgi]%\label{conj1}
	If   \eqref{ell1} is singular with $\lambda=1$, the space dimension $N\geq 3$, and %there is a non-negative measurable function $\Lambda(x)$ with$\lambda=1$, in which case we say that
	\begin{equation}\label{ell2}
		\io e^{\Lambda(x)}dx<\infty,
	\end{equation}
then weak solutions of \eqref{ell1} are continuous.
\end{conj}
This conjecture remains open. When $A=I+\m\otimes\m$ for some $\m\in \left(L^2(\Omega)\right)^N$ then we have %and it satisfies
$$|\xi|^2\leq A(x)\xi\cdot\xi=|\xi|^2+(\m\cdot\xi)^2\leq (1+|\m|^2)|\xi|^2\ \ \mbox{for a.e $x\in \Omega$ and each $\xi\in \rn$.}$$
One way of gaining a condition like \eqref{ell2} is to assume that $N=2$ and $\m\in \left(W^{1,2}_0(\Omega)\right)^2$.  Then  \eqref{ell2} is a consequence of Theorem 7.15 in \cite{GT}.
However, the case $N=2$ is not included in the above conjecture because the bare continuity in this case is rather trivial. We study the modulus of continuity when  our integrability condition on $\Lambda(x)=1+|\m|^2$ is much weaker.
%because the continuity of $u$ in this case does not seem to be difficult to come by. 
%At least in this case, a result of \cite{X5} (Also see Claim \ref{cpb} below.) asserts

%Here we develop a method based upon an inequality due to Fefferman and Stein \cite{FS}. To describe the method, we need to introduce some notions first. For each $r>0$ and $x_0\in\Omega$ such that_{\br\subset\Omega}
%$\br\subset\Omega$,
%Here and in what follows the letter $c$ denotes a generic positive constant. As usual, here $\br$ is the ball centered at $x_0$ with radius $r$
%,Here and in what follows and $\sup$ means ess sup and $\inf$ means ess inf.
%The definitions in \eqref{hc1} make sense because a result in \cite{LX} asserts that for each $q>\frac{N}{2}$ there is a positive number $c
%$ such that
It is well known that when $\Lambda(x)\leq c\lambda(x)$ for some $c>0$ and $\lambda(x)$ is an $A_2$ weight \cite{S} weak solutions of \eqref{ell1} are H\"{o}lder continuous no matter what the space dimensions are \cite{HKM}. The maximum H\"{o}lder exponent was investigated in \cite{PS} under the assumptions that $N$=2 and $A$ is bounded and uniformly elliptic.   %Under the assumptions of Conjecture \ref{conj1} equation \eqref{ell1} is said to be singular.$\Lambda$ and $\lambda$ are both positive constants
% If condition \eqref{ell3} is replaced by
%$$\frac{1}{\Lambda(x)}|\xi|^2\leq A(x)\xi\cdot\xi\leq |\xi|^2\ \ \mbox{for a.e $x\in \Omega$ and each $\xi\in \rn$,}$$
%then \eqref{ell1} is called degenerate. We refer the reader to \cite{OZ} for a continuity result in this case. 
If \eqref{ell1} is degenerate with $\Lambda=1$, $N=2$, and
$$K\equiv\io e^{\frac{\gamma}{\sqrt{\lambda(x)}}}dx<\infty\ \ \mbox{for some $\gamma>1$},$$
a result of \cite{OZ} asserts that for each $\beta\in (0,\gamma-1)$ there is a $c$ such that
$$\underset{\br}{\mbox{osc}}\ u\leq \frac{c\left(\io A\nabla u\cdot\nabla udx\right)^{\frac{1}{2}}}{\ln^{\frac{\beta}{2}}\frac{K}{\pi r^2}
	}\ \ \mbox{for $r$ sufficiently small}.$$
In view of this, our result (C3) is really surprising because the power $\ell$ in \eqref{hc2} can be arbitrarily large and the entries of $A$ are only square-integrable under (H3). It is natural for us to make the following
\begin{conj}Let (H1)-(H4) hold. Then the weak solution $p$ to \eqref{ell0}-\eqref{el1} is H\"older continuous. 
	\end{conj} 
%further conjecture that $p$ is H\"older continuous under (H1)-(H3).  
We point out that (H4) is not enough to guarantee \eqref{ell2} because $|\m|^2$ may not belong to $\textup{BMO}(\Omega)$. If the preceding conjecture were true, we would immediately be able to find applications for it in the mathematical analysis of the biological network formation model \cite{X7}. Unfortunately, we have not been able to prove or disprove the conjecture. However, we do have
\begin{thm}\label{part}Let (H1)-(H4) hold. Then whenever $x_0\in\Omega$ is such that 
	\begin{equation}\label{mav}
		 \sup_{r>0} |\m_{x_0,r}|<\infty,
	\end{equation}
	 where
	 $$\m_{x_0,r}=\frac{1}{|\br|}\int \m dx=\dashint_{\br} \m dx.$$	
	 we can find two positive numbers $c, \alpha$ with
	\begin{equation}\label{hc3}
		\underset{\br}{\textup{osc}}\ p\leq cr^\alpha\ \ \mbox{for $r$ sufficiently small.}\nonumber
	\end{equation}
	\end{thm}
%By virtue of Lemma \ref{logb} below, for each $\ve>0$ we have if some additional assumptions are made on $\m$ 
%$$\lim_{r\ra 0}r^\ve|\m_{x_0,r}|=0.$$ 
This is in essence a partial regularity result \cite{X6}. The so-called singular set
$$S\equiv\{x_0\in\Omega:\limsup_{r\ra 0}|\m_{x_0,r}|=\infty\}$$
can be estimated in terms of the Hausdorff measure \cite{G}.

In \cite{K} the relationship between the continuity of $p$ and $|\nabla p|^2\in K^{\mbox{loc}}_2(\Omega)$ was investigated. It turns out that the former implies the latter when  $A$ is bounded and uniformly elliptic.
Here, due to the singularity in our problem, the continuity of $p$ is not enough to guarantee \eqref{ko5}. %This is different from the result in \cite{K}. %Our main contribution in this study is that for each $\ell\geq 1$ there are  positive numbers $c, R$ with $B_R(x_0)\subset\Omega$ such that
%\begin{equation}\label{log1}
%	\underset{\br}{\mbox{osc}}p\leq \frac{c}{\ln^\ell\frac{R}{r}}\ \ \mbox{for all $r\in (0, R]$},
%\end{equation}
% where $\ell$ can 
%That is, the power $\ell$ can be arbitrarily large. This is surprising in view of the result in \cite{OZ}.
% we have managed to improve the modulus of continuity of $p$. 
As we shall see below, 
we must have $\ell>3$ in \eqref{hc2} to ensure that \eqref{ko5} holds. Of course, the most interesting result associated with the space $K^{\mbox{loc}}_2(\Omega)$ is the following
\begin{lem}\label{kato3}There is a constant $c$ such that
	\begin{equation}\label{kato5}
		\ibr|F|v^2dx\leq c\eta(F;\br;r)\ibr|\nabla v|^2dx\ \ \mbox{for each $v\in W^{1,2}_0(\br)$.}\nonumber
	\end{equation}
\end{lem}
This lemma is not really new, and it is only slightly different from Lemma 1.1 in \cite{FGL}, which has had many applications in the study of partial differential equations \cite{FGL,K}. We include the lemma here in the hope that this version of the inequality may be more suitable for applications in certain cases \cite{CFG,X7}.

We refer the reader to \cite{X7} for an application of (C3), (C4), and Lemma \ref{kato3}.

The rest of the paper is organized as follows. The proof of Theorem \ref{main} is given in Section \ref{s2}. It relies on a decomposition of the logarithmic function. Theorem \ref{part} and Lemma \ref{kato3} will be established in Section \ref{s3}.
\section{Proof of Theorem \ref{main}}\label{s2}
%In this section we collect some interesting results about that are relevant to PDE applications. Of course, this is done with our problem in mind. But not all of them are directly applicable to our problem.  Much more information can be found in \cite{SE}. 

%It is enough for us to 
In this section we offer the proof of Theorem \ref{main}. The proof is divided into several lemmas.

\begin{lem}Let (H1)-(H3) hold and $p$ be the weak solution to \eqref{ell0}-\eqref{el1}.
	%Fix $y\in \Omega$ and set%so that $B_{2R}(x_0)\subset\Omega$.\in\Omega_{0}$
%	Let $\Omega_0$ be a sub-domain of $\Omega$ such that+c(1+|\ln r|)\int_{B_r(y)}|\nabla p|^2dx
	%$$d =\min\left\{1, \mbox{dist}(y,\partial\Omega)\right\}.$$
	Then there is a positive constant $c$ such that 
	\begin{eqnarray}
	%\lefteqn{\sup_{y\in\brh}	\ibrh  F\left|\ln|x-y|\right| dx}\nonumber\\
	\ibRy F\left|\ln|x-y|\right| dx&\leq&(\ln2+|\ln r|) 	\ibRy  F dx+c \ibRy|\nabla p|^2dx\nonumber\\
	%	&\leq&(\ln2+|\ln r|)	\int_{ B_{\frac{3r}{2}}(x_0)}   F dx+c\int_{ B_{\frac{3r}{2}}(x_0)} |\nabla p|^2dx\nonumber\\
&&	+c\left(\ibRy|\m|^4dx\right)^{\frac{1}{2}}\ibRy |\nabla p|^2dx+cr^{\frac{2(q-1)}{q}}\| S(x)\|_{q,\bRy}\label{wish1}
	%	&&+c\left(\int_{ B_{\frac{3r}{2}}(x_0)}|\m|^4dx\right)^{\frac{1}{2}}\int_{ B_{\frac{3r}{2}}(x_0)} |\nabla p|^2dx+cr^{\frac{2(q-1)}{q}}\| S(x)\|_{q, B_{\frac{3r}{2}}(x_0)}\label{wish1}
		%\ibRy  F\left|\ln|x-y|\right|^2 dx
	%	&\leq&c(1+\ln^2 r)	\ibRy   F dx
	%	+	c\int_{B_r(y)}|\ln|x-y|||\nabla p|^2dx\nonumber\\
	%	&&+c\left(\int_{B_r(y)}|\m|^4dx\right)^{\frac{1}{2}}\int_{B_r(y)}|\ln|x-y|||\nabla p|^2dx\nonumber\\
	%	&&+c(1+|\ln r|)\left(\int_{B_r(y)}|\m|^4dx\right)^{\frac{1}{2}}\int_{B_r(y)}|\nabla p|^2dx\nonumber\\
	%	&&+c(1+|\ln r|)r^{\frac{(2-\sigma)(q-1)}{q}}\| S(x)\|_{q, B_r(y)}\ \ \mbox{each $r\in\left(0, d\right)$}.\label{wish2}
			\end{eqnarray}
	%	for each  $x_0\in\Omega$ and $r\in \left(0, \frac{2}{3} \mbox{dist}(x_0,\partial\Omega)\right)$.
	for each  $y\in\Omega$ and $r\in \left(0,  \mbox{dist}(y,\partial\Omega)\right)$.
\end{lem}

\begin{proof}Let $y, r$ be given as in the lemma.  Pick a smooth function $\zeta(x)$ on $\rn$ such that $\zeta(x)=1$ on $B_1(0)$, $0\leq \zeta(x)\leq 1$ on $\rn$, and $\zeta(x)=0$ when $|x|\geq 2$. Set
	\begin{equation}\label{ldef}
		l(x)=-\sum_{k=1}^{\infty}\zeta^2\left(\frac{2^k(y-x)}{r}\right).
	\end{equation}
	Then $\ln\frac{|y-x|}{r}-l(x)\ln 2	\in L^\infty(B_r(y))$. Indeed, for each $x\in B_r(y)\setminus\{y\}$ there must exist a $j\in\{1,2,\cdots\}$ such that
	$$\frac{1}{2^j}\leq\frac{|y-x|}{r}<\frac{1}{2^{j-1}}.$$
	We easily verify from the properties of $\zeta$ that
	$$\zeta^2\left(\frac{2^k(y-x)}{r}\right)=\left\{\begin{array}{cc}
		0&\mbox{if $k>j$,}\\
		1&\mbox{if $k\leq j-1$.}
	\end{array}\right.$$
That is, the series in \eqref{ldef} is actually a finite sum and
	$$-l(x)=j-1+\zeta^2\left(\frac{2^j(y-x)}{r}\right).$$
	It immediately follows that
	$$j-1\leq -l(x)\leq j,\ \ -j\ln2\leq\ln\frac{|y-x|}{r}\leq -(j-1)\ln2.$$
	Hence,
	\begin{eqnarray}
		\ln\frac{|y-x|}{r}-l(x)\ln 2&\leq &\ln\frac{|y-x|}{r}+j\ln 2\leq \ln2,\label{lb1}\\
		\ln\frac{|y-x|}{r}-l(x)\ln 2&\geq &	\ln\frac{|y-x|}{r}+(j-1)\ln 2\geq -\ln2.\label{lb2}\nonumber
	\end{eqnarray}
	That is,
	\begin{equation}\label{wm5}
		\left|\ln\frac{|y-x|}{r}-l(x)\ln 2\right|\leq \ln2.
	\end{equation}
We would like to remark that this type of decomposition of the logarithmic function is often found in the study of real variable Hardy spaces \cite{SE}.
 
	Let
	$$A_j=B_{\frac{r}{2^{j-1}}}(y)\setminus B_{\frac{r}{2^{j}}}(y).$$
Then
	$$B_r(y)=\cup_{j=1}^{\infty}A_j.$$
	Denote by $p_{A_j}$ the average of $p$ over $A_j$, i.e.,
	$$p_{A_j}=\dashint_{A_j}pdx=\frac{1}{|A_j|}\int_{A_j}pdx.$$
	Subsequently, we can infer from the Sobolev-Poincar\'{e} inequality (\cite{GT}, p.174) that for each $s\geq 2$ there is a positive number $c$ such that
	\begin{equation}\label{spin}
		\left(\dashint_{A_j}|p-p_{A_j}|^sdx\right)^{\frac{1}{s}}	\leq  \frac{cr}{2^j}\left(\dashint_{A_j}|\nabla p|^{\frac{2s}{s+2}}dx\right)^{\frac{s+2}{2s}}\leq  \frac{cr}{2^j}\left(\dashint_{A_j}|\nabla p|^{2}dx\right)^{\frac{1}{2}}.
	\end{equation}
	%$$$
	%With this in mind, we 
With this in mind, we use $(p-p_{A_j})  \zeta^2\left(\frac{2^j(y-x)}{r}\right)$ as a test function in \eqref{ell0} 
	%and remember
	%	$$h=p\ \ \mbox{on $\{\theta\ne 0\}$}$$
	to derive
	\begin{eqnarray}
		\lefteqn{\ibryj\left(|\nabla p|^2  +\mnp^2  \right)\zeta^2\left(\frac{2^j(y-x)}{r}\right) dx}\nonumber\\
		&=&%-\frac{r}{r^N}\int_{A_j}\nabla p\cdot\nabla  (p-p_{A_j})\zeta\left(\frac{2^j(y-x)}{r}\right)dx\nonumber\\
		\frac{2^{j+1}}{r}\int_{A_j}\zeta\left(\frac{2^j(y-x)}{r}\right)\nabla p\cdot\nabla\zeta\left(\frac{2^j(y-x)}{r}\right)(p-p_{A_j})   dx\nonumber\\
		%	&&-	\ibrxj\nabla p\cdot\nabla\theta \zeta^2\left(\frac{2^j(y-x)}{r}\right)(p-p_{A_j})  dx\nonumber\\
		%&&-\frac{r}{r^N}\int_{A_j}\mnp\m\cdot\nabla  (p-p_{A_j})\zeta\left(\frac{2^j(y-x)}{r}\right)dx\nonumber\\\zeta\left(\frac{2^j(y-x)}{r}\right)
		&&+\frac{2^{j+1}}{r}\int_{A_j}\zeta\left(\frac{2^j(y-x)}{r}\right)\mnp\m\cdot\nabla\zeta\left(\frac{2^j(y-x)}{r}\right)(p-p_{A_j})  dx\nonumber\\
		%	&&-\ibrxj\mnp\m\cdot\nabla\theta\zeta^2\left(\frac{2^j(y-x)}{r}\right)(p-p_{A_j})   dx\nonumber\\
		&&+\ibryj S(y)(p-p_{A_j})  \zeta^2\left(\frac{2^j(y-x)}{r}\right) dx.\label{wm4}
	\end{eqnarray}
	We proceed to estimate each term on the right hand side. First, by \eqref{spin}, we have
	\begin{eqnarray}
		\lefteqn{\frac{2^{j+1}}{r}\int_{A_j}\zeta\left(\frac{2^j(y-x)}{r}\right)\nabla p\cdot\nabla\zeta\left(\frac{2^j(y-x)}{r}\right)(p-p_{A_j})  dx}\nonumber\\
		&\leq&\frac{\|\nabla\zeta\|_{\infty,B_2(0)}2^{j+1}}{r}\left(\int_{A_j}|\nabla p|^2dx\right)^{\frac{1}{2}}\left(\int_{A_j}|p-p_{A_j}|^2dx\right)^{\frac{1}{2}}\leq c\int_{A_j}|\nabla p|^2dx.\nonumber
	\end{eqnarray}
	Similarly, for each $\ve>0$ we calculate
	\begin{eqnarray}
		\lefteqn{\frac{2^{j+1}}{r}\int_{A_j}\zeta\left(\frac{2^j(y-x)}{r}\right)\mnp\m\cdot\nabla\zeta\left(\frac{2^j(y-x)}{r}\right)(p-p_{A_j}) dx	}\nonumber\\
		&\leq&\ve\ibryj\mnp^2 \zeta^2\left(\frac{2^j(y-x)}{r}\right)dx
		+\frac{\|\nabla\zeta\|_{\infty,B_2(0)}^24^{j+1}}{\ve r^2}\int_{A_j}|\m|^2|p-p_{A_j}|^2dx\nonumber\\
		&\leq&\ve\ibryj\mnp^2 \zeta^2\left(\frac{2^j(y-x)}{r}\right) dx
		+\frac{c4^{j}}{\ve r^2}\left(\int_{A_j}|\m|^4dx\right)^{\frac{1}{2}}\left(\int_{A_j}|p-p_{A_j}|^4dx\right)^{\frac{1}{2}}\nonumber\\
		&\leq&\ve\ibryj\mnp^2 \zeta^2\left(\frac{2^j(y-x)}{r}\right)dx
		+\frac{c}{\ve}\left(\int_{B_r(y)}|\m|^4dx\right)^{\frac{1}{2}}\int_{A_j}|\nabla p|^2dx.\nonumber
		%	&\leq&c\left(\int_{A_j}\mnp^2|\m|^2dx\right)^{\frac{1}{2}}\left(\int_{A_j}|\nabla p|^2dx\right)^{\frac{1}{2}}\nonumber\\
		%	&\leq&c\int_{A_j}\mnp^2|\m|^2dx+ c.
	\end{eqnarray}
Finally,
	\begin{eqnarray}
		\ibryj S(x)(p-p_{A_j})  \zeta^2\left(\frac{2^j(y-x)}{r}\right) dx&\leq&c\| S(x)\|_{q, B_r(y)}\left(\frac{r}{2^{j-1}}\right)^{\frac{2(q-1)}{q}}.\nonumber
	\end{eqnarray}
Substitute the preceding three estimates into \eqref{wm4} and choose $\ve$ suitably small in the resulting inequality to obtain
\begin{eqnarray}
		\ibryj F\zeta^2\left(\frac{2^j(y-x)}{r}\right) dx&=&	\ibryj\left(|\nabla p|^2  +\mnp^2  \right)\zeta^2\left(\frac{2^j(y-x)}{r}\right) dx\nonumber\\
	&\leq&c\int_{A_j}|\nabla p|^2dx+c\left(\int_{B_r(y)}|\m|^4dx\right)^{\frac{1}{2}}\int_{A_j}|\nabla p|^2dx\nonumber\\
	&&+c\| S(x)\|_{q, B_r(y)}\left(\frac{r}{2^{j-1}}\right)^{\frac{2(q-1)}{q}}.\label{wm6}
\end{eqnarray}
With this and \eqref{wm5} in mind, we derive
	\begin{eqnarray}
		\ibRy  F\left|\ln|x-y|\right| dx
	&\leq&	\ibRy   F |\ln|x-y|-\ln r-l(x)\ln2|dx\nonumber\\
	&&+|\ln r|	\ibRy   F dx+\ln 2\ibRy   F |l(x)|dx\nonumber\\
		&\leq&(\ln2+|\ln r|)	\ibRy   F dx+\ln 2\ibRy   F \sum_{j=1}^{\infty}\zeta^2\left(\frac{2^j(y-x)}{r}\right)dx\nonumber\\
		&\leq&(\ln2+|\ln r|)	\ibRy   F dx+c\ibRy |\nabla p|^2dx\nonumber\\
		&&+c\left(\int_{B_r(y)}|\m|^4dx\right)^{\frac{1}{2}}\ibRy |\nabla p|^2dx+cr^{\frac{2(q-1)}{q}}\| S(x)\|_{q, B_r(y)}.\nonumber
	\end{eqnarray}
%This combined with \eqref{wi1} and \eqref{wi2} completes the proof of \eqref{wish1}.
The lemma follows.
	\end{proof}
	%If $n=2$, then
	\begin{lem}Let the assumptions of the preceding lemma hold.
		%$x_0\in\Omega$.
		%Fix $y\in \Omega$ and set%so that $B_{2R}(x_0)\subset\Omega$.\in\Omega_{0}$
		%	Let $\Omega_0$ be a sub-domain of $\Omega$ such that+c(1+|\ln r|)\int_{B_r(y)}|\nabla p|^2dx
		%$$d =\min\left\{1, \mbox{dist}(y,\partial\Omega)\right\}.$$
		Then there is a positive constant $c$ such that 
		\begin{eqnarray}
		\lefteqn{	\ibRy  F\ln^2|x-y| dx}\nonumber\\
		%	\int_{ B_{\frac{3r}{2}}(x_0)} F\left|\ln|x-y|\right|^2 dx
				&\leq&c(1+\ln^2 r)	\int_{B_r(y)}  F dx
				+	c\int_{B_r(y)}|\ln|x-y|||\nabla p|^2dx\nonumber\\
				&&+c\left(\int_{B_r(y)}|\m|^4dx\right)^{\frac{1}{2}}\int_{B_r(y)}|\ln|x-y|||\nabla p|^2dx\nonumber\\
				&&+c(1+|\ln r|)\left(\int_{B_r(y)}|\m|^4dx\right)^{\frac{1}{2}}\int_{B_r(y)}|\nabla p|^2dx\nonumber\\
				&&+c(1+|\ln r|)r^{\frac{(2-\sigma)(q-1)}{q}}\| S(x)\|_{q, B_r(y)}.\label{wish2}
		\end{eqnarray}
		for each $y\in\Omega$ and $r\in \left(0, \mbox{dist}(y,\partial\Omega)\right)$.
	\end{lem}
	\begin{proof}
		We see from \eqref{wm5} that
	\begin{eqnarray}
	\ibRy  F\left|\ln|x-y|\right|^2 dx	&\leq&2	\ibRy   F |\ln|x-y|-\ln r-l(x)\ln2|^2dx\nonumber\\
		&&+4\ln^2 r	\ibRy   F dx+4\ln^2 2\ibRy   F |l(x)|^2dx\nonumber\\
		&\leq&(2\ln^22+4\ln^2 r)	\ibRy   F dx+4\ln^2 2\ibRy   F |l(x)|^2dx.\label{wm10}
	\end{eqnarray}
It is easy to verify that 
$$\zeta^2\left(\frac{2^j(y-x)}{r}\right)\zeta^2\left(\frac{2^i(y-x)}{r}\right)=\zeta^2\left(\frac{2^i(y-x)}{r}\right)\ \ \mbox{whenever $i<j$.}$$
Consequently,
\begin{eqnarray}
	|l(x)|^2&=&\sum_{i=1}^{\infty}\zeta^2\left(\frac{2^i(y-x)}{r}\right)\sum_{j=1}^{\infty}\zeta^2\left(\frac{2^j(y-x)}{r}\right)\nonumber\\
	&=&\sum_{j=1}^{\infty}\zeta^4\left(\frac{2^j(y-x)}{r}\right)+2\sum_{j=1}^{\infty}(j-1)\zeta^2\left(\frac{2^j(y-x)}{r}\right).\label{wm9}
\end{eqnarray}
%Use $(p-p_{A_j})  \zeta^4\left(\frac{2^j(y-x)}{r}\right)$ as a test function in \eqref{ell0} and argue as 
We can infer from the proof of \eqref{wm6} that
\begin{eqnarray}
	\lefteqn{\ibryj\left(|\nabla p|^2  +\mnp^2  \right)\zeta^4\left(\frac{2^j(y-x)}{r}\right) dx}\nonumber\\
	&\leq&c\int_{A_j}|\nabla p|^2dx+c\left(\int_{B_r(y)}|\m|^4dx\right)^{\frac{1}{2}}\int_{A_j}|\nabla p|^2dx+c\| S(x)\|_{q, B_r(y)}\left(\frac{r}{2^{j-1}}\right)^{\frac{2(q-1)}{q}}.\label{wm7}
\end{eqnarray}
%and remember
%	$$h=p\ \ \mbox{on $\{\theta\ne 0\}$}$$
Use $(p-p_{A_j}) (j-1)\zeta^2\left(\frac{2^j(y-x)}{r}\right)$ as a test function in \eqref{ell0}
to derive
\begin{eqnarray}
	\lefteqn{\ibryj\left(|\nabla p|^2  +\mnp^2  \right)(j-1)\zeta^2\left(\frac{2^j(y-x)}{r}\right) dx}\nonumber\\
	&=&%-\frac{r}{r^N}\int_{A_j}\nabla p\cdot\nabla  (p-p_{A_j})\zeta\left(\frac{2^j(y-x)}{r}\right)dx\nonumber\\
	\frac{2^{j+1}}{r}\int_{A_j}(j-1)\zeta\left(\frac{2^j(y-x)}{r}\right)\nabla p\cdot\nabla\zeta\left(\frac{2^j(y-x)}{r}\right)(p-p_{A_j})   dx\nonumber\\
	&&+\frac{2^{j+1}}{r}\int_{A_j}(j-1)\zeta\left(\frac{2^j(y-x)}{r}\right)\mnp\m\cdot\nabla\zeta\left(\frac{2^j(y-x)}{r}\right)(p-p_{A_j})  dx\nonumber\\
	&&+\int_{B_{\frac{r}{2^{j-1}}}(y)} S(y)(p-p_{A_j})  (j-1)\zeta^2\left(\frac{2^j(y-x)}{r}\right)dx.\label{wm8}
\end{eqnarray}
Recall from \eqref{lb1} that
$$j-1\leq -\frac{1}{\ln2}\ln\frac{|x-y|}{r}.$$
%We proceed to estimate each term on the right hand side. First, by \eqref{spin}, 
Keep this in mind to deduce for each $\ve>0$ that
\begin{eqnarray}
	\lefteqn{\frac{2^{j+1}}{r}\int_{A_j}(j-1)\zeta\left(\frac{2^j(y-x)}{r}\right)\nabla p\cdot\nabla\zeta\left(\frac{2^j(y-x)}{r}\right)(p-p_{A_j})  dx}\nonumber\\
	&\leq &\ve\ibryj|\nabla p|^2(j-1)\zeta^2\left(\frac{2^j(y-x)}{r}\right) dx+
\frac{\|\nabla\zeta\|_{\infty,B_2(0)}^24^{j+1}(j-1)}{r^2\ve}\int_{A_j}|p-p_{A_j}|^2dx\nonumber\\
%\leq c\int_{A_j}|\nabla p|^2dx
&\leq &\ve\ibryj|\nabla p|^2(j-1)\zeta^2\left(\frac{2^j(y-x)}{r}\right) dx+\frac{c}{\ve}\int_{A_j}(j-1)|\nabla p|^2dx\nonumber\\
&\leq &\ve\ibryj|\nabla p|^2(j-1)\zeta^2\left(\frac{2^j(y-x)}{r}\right) dx+\frac{c}{\ve}\int_{A_j}\left| \ln|x-y|-\ln r\right||\nabla p|^2dx.\nonumber
\end{eqnarray}
By the same token, 
%Similarly,  we calculate
\begin{eqnarray}
	\lefteqn{\frac{2^{j+1}}{r}\int_{A_j}(j-1)\zeta\left(\frac{2^j(y-x)}{r}\right)\mnp\m\cdot\nabla\zeta\left(\frac{2^j(y-x)}{r}\right)(p-p_{A_j}) dx	}\nonumber\\
%	&\leq&\ve\ibryj\mnp^2 j\zeta^2\left(\frac{2^j(y-x)}{r}\right)dx
%	+\frac{\|\nabla\zeta\|_{\infty,B_2(0)}^24^{j+1}j}{\ve r^2}\int_{A_j}|\m|^2|p-p_{A_j}|^2dx
	&\leq&\ve\ibryj\mnp^2 (j-1)\zeta^2\left(\frac{2^j(y-x)}{r}\right) dx\nonumber\\
	&&+\frac{c4^{j}(j-1)}{\ve r^2}\left(\int_{A_j}|\m|^4dx\right)^{\frac{1}{2}}\left(\int_{A_j}|p-p_{A_j}|^4dx\right)^{\frac{1}{2}}\nonumber\\
	&\leq&\ve\ibryj\mnp^2 (j-1)\zeta^2\left(\frac{2^j(y-x)}{r}\right)dx\nonumber\\
	&&+\frac{c}{\ve}\left(\int_{B_r(y)}|\m|^4dx\right)^{\frac{1}{2}}\int_{A_j}\left| \ln|x-y|-\ln r\right||\nabla p|^2dx.\nonumber
	%	&\leq&c\left(\int_{A_j}\mnp^2|\m|^2dx\right)^{\frac{1}{2}}\left(\int_{A_j}|\nabla p|^2dx\right)^{\frac{1}{2}}\nonumber\\
	%	&\leq&c\int_{A_j}\mnp^2|\m|^2dx+ c.
\end{eqnarray}
The last term in \eqref{wm8} can be estimated as follows:
\begin{eqnarray}
\lefteqn{	\ibryj S(x)(p-p_{A_j})  (j-1)\zeta^2\left(\frac{2^j(y-x)}{r}\right) dx}\nonumber\\
&\leq&c	\ibrxj |S(x)|\left|\ln|x-y|-\ln r\right|dx\nonumber\\
	&\leq&c|\ln r|\| S(x)\|_{q, B_r(y)}\left(\frac{r}{2^{j-1}}\right)^{\frac{2(q-1)}{q}}+c\| S(x)\|_{q, B_r(y)}\left(\frac{r}{2^{j-1}}\right)^{\frac{(2-\sigma)(q-1)}{q}},\ \ \sigma\in (0,2).\nonumber
\end{eqnarray}
Collect the  preceding three estimates in \eqref{wm8} to deduce
\begin{eqnarray}
	\lefteqn{\ibryj\left(|\nabla p|^2  +\mnp^2  \right)(j-1)\zeta^2\left(\frac{2^j(y-x)}{r}\right) dx}\nonumber\\
		&\leq&c\int_{A_j}|\ln|x-y|||\nabla p|^2dx+c\left(\int_{B_r(y)}|\m|^4dx\right)^{\frac{1}{2}}\int_{A_j}|\ln|x-y|||\nabla p|^2dx\nonumber\\
	&&+c|\ln r|\int_{A_j}|\nabla p|^2dx+c|\ln r|\left(\int_{B_r(y)}|\m|^4dx\right)^{\frac{1}{2}}\int_{A_j}|\nabla p|^2dx\nonumber\\
	&&+c(1+|\ln r|)\| S(x)\|_{q, B_r(y)}\left(\frac{r}{2^{j-1}}\right)^{\frac{(2-\sigma)(q-1)}{q}}.\nonumber
\end{eqnarray}
Combining this with \eqref{wm7} and \eqref{wm9} yields
%With this and \eqref{wm5} in mind, we derive
\begin{eqnarray}
\lefteqn{	\ibRy  F\left|l^2(x)\right| dx}\nonumber\\
	&\leq&	\ibRy   F\sum_{j=1}^{\infty}\zeta^4\left(\frac{2^j(y-x)}{r}\right)dx+2\ibRy   F \sum_{j=1}^{\infty}(j-1)\zeta^2\left(\frac{2^j(y-x)}{r}\right)dx\nonumber\\
	&\leq&	c\int_{B_r(y)}|\ln|x-y|||\nabla p|^2dx+c\left(\int_{B_r(y)}|\m|^4dx\right)^{\frac{1}{2}}\int_{B_r(y)}|\ln|x-y|||\nabla p|^2dx\nonumber\\
	&&+c(1+|\ln r|)\int_{B_r(y)}|\nabla p|^2dx+c(1+|\ln r|)\left(\int_{B_r(y)}|\m|^4dx\right)^{\frac{1}{2}}\int_{B_r(y)}|\nabla p|^2dx\nonumber\\
	&&+c(1+|\ln r|)r^{\frac{(2-\sigma)(q-1)}{q}}\| S(x)\|_{q, B_r(y)}.\nonumber
\end{eqnarray}
Use this in \eqref{wm10} to obtain \eqref{wish2}. The proof is complete.
%\sum_{j=1}^{\infty}\zeta^2\left(\frac{2^j(y-x)}{r}\right)
\end{proof}

We can also verify that
\begin{eqnarray}
	|l(x)|^3&=&\sum_{j=1}^{\infty}\zeta^6\left(\frac{2^j(y-x)}{r}\right)+2\sum_{j=1}^{\infty}(j-1)\zeta^4\left(\frac{2^j(y-x)}{r}\right)+\sum_{j=1}^{\infty}(j-1)\zeta^2\left(\frac{2^j(y-x)}{r}\right)\nonumber\\
	&&+\sum_{j=1}^{\infty}(j-1)^2\zeta^2\left(\frac{2^j(y-x)}{r}\right)+\sum_{j=2}^{\infty}\frac{(j-1)(j-2)}{2}\zeta^2\left(\frac{2^j(y-x)}{r}\right).\nonumber
\end{eqnarray}
It is not difficult to deduce that for each positive integer $n$ we have
$$|l(x)|^n=\sum_{j=1}^{\infty}\zeta^{2n}\left(\frac{2^j(y-x)}{r}\right)+\mbox{lower order terms}.$$
Our earlier proof indicates that
$$\ibRy F|\ln|x-y||^{n-1}dx<\infty\implies \ibRy F|\ln|x-y||^{n}dx<\infty.$$
In summary, we have:
\begin{lem}Let the assumptions of the preceding lemma hold. For each positive integer $n, y\in\Omega, r\in (0, \mbox{dist}(y, \partial\Omega))$ there is a positive number $c$ such that
	\begin{equation}\label{fln}
		\ibRy F|\ln|x-y||^{n}dx\leq c.
	\end{equation}
	\end{lem}
%\begin{lem}Let $y\in\Omega$. Then for each $\ell>1, \ R\in \left(0,\min\{1, \mbox{dist}(y,\partial\Omega)\}\right)$ there is a positive number $c$ such that
%	\begin{equation}
%			\begin{split}
	%			\underset{ B_  \tau(y)}{	\textup{osc}}p
				%=\sup_{ B_  tau(y)}p-\inf_{ B_  tau(y)}p\\
				%	=&\sup_{x_1, x_2\in B_r(y)}(p(x_1,t)-p(x_2,t))\\
				%	=&\sup_{x_1, x_2\in B_r(y)}(p_0(x_1,t)-p_0(x_2,t)+p_1(x_1,t)-p_1(x_2,t))\\
				%\leq &\ \underset{ B_  \tau(y)}{	\mbox{osc}}p_0+\underset{ B_  \tau(y)}{	\mbox{osc}}p_1\\
		%		\leq & \frac{c}{\ln^{\ell+1}\frac{R}{\tau}} +c  \tau^{\frac{2(q-1)}{q}} \ \ \mbox{for each $\tau\in (0, R)$}.
		%	\end{split}
	%	\end{equation}

%\end{lem}
We are ready to prove (C3).
\begin{proof}[Proof of (C3)] Let $y\in\Omega, R\in \left(0,\min\{1, \mbox{dist}(y,\partial\Omega)\}\right)$ be given.
%	Fix a point $y$ in $\Omega$. 
%Note that $\m\in C([0,T]; \left(L^2(\Omega)\right)^N)$. each $t\in (0, T]$ and
For  each $r\in (0, R]$,
	we consider the boundary value problem
	\begin{align}
		-\mbox{{div}}\left[(I+\m\otimes \m)\nabla p_1\right]&=S(x)\ \ \ \mbox{in $B_r(y)$,}\label{do1}\\
		p_1&=0\ \ \ \mbox{on $\partial B_r(y)$.}\label{do2}
	\end{align}
	Even though the elliptic coefficients in \eqref{do1} may not be bounded, we can easily infer from the proof of Lemma 2.3 in \cite{X3} that 
	this problem has a unique solution $p_1=p_1(x,t)$ in the sense of (D1)-(D2).
	
\begin{clm}\label{cpb}For each $q>1$ there is a positive  number $c=c(N, q)$ such that
	\begin{equation}\label{pb2}
		\underset{\bRy}{\textup{ess sup}}\ |p_1|\leq cr^{\frac{2(q-1)}{q}}\left(\int_{B_r(y)}|S(x)|^{q}dx\right)^{\frac{1}{q}}.
	\end{equation}
%	Here and in what follows $\sup$ (resp. $\inf$) means $\mbox{ess sup}$ (resp. $\mbox{ess inf}$). \underset{\bRy}{\textup{ess sup}}
	\end{clm}% Furthermore, we are in a position to assert from (H1) and Proposition 2.1 in \cite{LX} that
\begin{proof} This result was established in \cite{LX} for $N>2$. The proof is similar for $N=2$. For completeness, we offer a proof here.
	We employ the classical Moser-Nash-De Giorgi type of arguments. Without loss of generality, we assume
	\begin{equation}
		\underset{\bRy}{\textup{ess sup}}\ p_1=\underset{\bRy}{\textup{ess sup}}\ |p_1|.\nonumber
	\end{equation}Let $\kappa$ be a positive number to be determined. Write
	$$\kappa_n=\kappa-\frac{\kappa}{2^n},\ \ \ A_n =\{x\in B_r(y): p_1(x,t)\geq\kappa_n\},\ \ \ n=0, 1,2,\cdots.$$
	Use $(p_1-\kappa_n)^+$ as a test function in (\ref{do1}) to deduce
	\begin{equation}\label{pb1}
		\int_{B_r(y)}|\nabla (p_1-\kappa_n)^+|^2dx+\int_{B_r(y)}(\m\cdot\nabla(p_1-\kappa_n)^+)^2dx=\int_{B_r(y)} S(x)(p_1-\kappa_n)^+dx.
	\end{equation}
Here we have used the fact that
$$\nabla p_1\chi_{A_n }=\nabla(p_1-\kappa_n)^+.$$ 
By Poincar\'{e}'s inequality,
	\begin{eqnarray}
	%	\lefteqn{}\nonumber\\\frac{2N}{N+2}
	%	&=&\nonumber\\
	\int_{B_r(y)} S(x)(p_1-\kappa_n)^+dx	&\leq&\left(\int_{B_r(y)}|S(x)|^{q}\right)^{\frac{1}{q}}\left(\int_{B_r(y)}\left[(p_1-\kappa_n)^+\right]^{\frac{q}{q-1}}dx\right)^{\frac{q-1}{q}}\nonumber\\
		&\leq&c\|S(x)\|_{q,B_r(y) }\left(\int_{B_r(y)}\left|\nabla(p_1-\kappa_n)^+\right|^{\frac{2q}{3q-2}}dx\right)^{\frac{3q-2}{2q}}\nonumber\\
		&\leq&	c\|S(x)\|_{q,B_r(y) }\left(\int_{B_r(y)}\left|\nabla(p_1-\kappa_n)^+\right|^{2}dx\right)^{\frac{1}{2}}|A_n |^{\frac{q-1}{q}}.\nonumber
	\end{eqnarray}
	Use this in \eqref{pb1} to obtain
	\begin{equation}
	\left(\int_{B_r(y)}\left|\nabla(p_1-\kappa_n)^+\right|^{2}dx\right)^{\frac{1}{2}}	\leq 	c\|S(x)\|_{q,B_r(y) }|A_n |^{\frac{q-1}{q}}.\nonumber
	\end{equation}
For each $s>1$ we have
\begin{eqnarray}
	\frac{\kappa}{2^{n+1}}|A_{n+1} |^{\frac{1}{s}}&\leq&\left(\int_{B_r(y)}\left[(p_1-\kappa_n)^+\right]^{s}dx\right)^{\frac{1}{s}}\nonumber\\
	&\leq&c\left(\int_{B_r(y)}\left|\nabla(p_1-\kappa_n)^+\right|^{\frac{2s}{s+2}}dx\right)^{\frac{2+s}{2s}}\nonumber\\
		&\leq&	c\left(\int_{B_r(y)}\left|\nabla(p_1-\kappa_n)^+\right|^{2}dx\right)^{\frac{1}{2}}r^{\frac{2}{s}}\nonumber\\
		&\leq&	cr^{\frac{2}{s}}\|S(x)\|_{q,B_r(y) }|A_n |^{\frac{q-1}{q}},\nonumber
\end{eqnarray}
	from whence follows
	\begin{equation}
		|A_{n+1} |\leq cr^2\|S(x)\|_{q,B_r(y) }^{s}\frac{2^{(n+1)s}}{\kappa^{s}}|A_{n} |^{\frac{(q-1)s}{q}}.\nonumber
	\end{equation}
Now we pick $s$ so large that
$$\alpha\equiv \frac{(q-1)s}{q}-1=\frac{s(q-1)-q}{q}>0.$$
	%By (H1), we have $$. This enables us to apply 
	We are in a position to apply Lemma 4.1 in (\cite{D}, p. 12). To this end, we choose $\kappa$ so large that
	$$A_0 \leq |B_r(y)|\leq c\left(\frac{\kappa^{s}}{r^2\|S(x)\|_{q,B_r(y) }^{s}}\right)^{\frac{q}{s(q-1)-q}}.$$
	% to obtain$$|A_\infty |=0,\ \ \ \mbox{provided that $\kappa=c\|S(x)\|_q$ for some $c=c(\Omega, N)$.}$$
	Then
	$$p_1\leq \kappa\ \ \mbox{on $B_r(y)$.}$$
	It is enough for us to take
	$$\kappa= cr^{\frac{2(q-1)}{q}}\|S(x)\|_{q,B_r(y) }.$$
The lemma follows. The proof is complete.
\end{proof}	

	Obviously, $p_0\equiv p-p_1$ satisfies
	%is the unique solution of the boundary value problem
	\begin{align}
		-\mbox{{div}}\left[(I+\m\otimes \m)\nabla p_0\right]&=0\ \ \ \mbox{in $B_r(y)$,}\nonumber\\
		p_0(x,t)&=p(x,t)\ \ \ \mbox{on $\partial B_r(y)$}\nonumber
	\end{align}
	in the sense of (D1)-(D2) with an obvious modification to the boundary condition. 
	%That is,  we can decompose $p(x,t)$ into the sum of $p_0(x,t)$ and $p_1(x,t)$ on $B_r(y)$, or equivalently,
%	\begin{equation}
	%	p=p_0+p_1\ \ \ \mbox{on $B_r(y)$}.
%	\end{equation} Set $k_l=\sup_{\partial B_r(y)}p$. Then $(p_0-k_l)^+\in W_0^{1,2}(B_r(y))$ with $m\cdot\nabla(p_0-k_l)^+ \in L^2(B_r(y))$. Thus we can use it as a test function in \eqref{dz1}, thereby obtaining
%	$$p_0\leq k_l \ \ \mbox{a.e. on $B_r(y)$.}$$
%	In fact, we can further conclude that
	%It is not difficult to derive from \eqref{dz1} that 
%	$p_0$ is a weakly monotone function \cite{M}, i.e.,
%	\begin{equation}
	%	\sup_{\partial \Omega^\prime}p_0=\sup_{ \Omega^\prime}p_0\ \ \ \mbox{and }\ \ \ \inf_{\partial \Omega^\prime}p_0=\inf_{ \Omega^\prime}p_0
%	\end{equation}
%	for each sub-domain $\Omega^\prime$ of $B_r(y)$. By Morrey's inequality on spheres as formulated by Gehring \cite{G1}, we obtain
Obviously, the weak maximum principle still holds, from which it follows that
	\begin{equation}
	\underset{ B_r(y)}{	\mbox{osc}} p_0=\underset{\partial  B_r(y)}{	\mbox{osc}}	p_0=\underset{ \partial B_r(y)}{	\mbox{osc}}p\leq \int_{\partial B_r(y)}|\nabla p|ds.\nonumber
	\end{equation}
(Here we may assume that $p$ is a smooth function because $p$ can be viewed as the limit of a sequence of smooth functions.)
For each $\ell>1$ and $ \tau\in (0,R)$ we multiply through the above inequality by $\frac{|\ln r|^\ell}{r}$ and then integrate the resulting one over $[\tau, R]$  to derive
\begin{eqnarray}
	\int_{\tau}^{R}\frac{	|\ln r|^\ell}{r}\underset{ B_r(y)}{	\mbox{osc}}	p_0dr&\leq &\int_{\tau}^{R}\frac{|\ln r|^\ell}{r}\int_{\partial B_r(y)}|\nabla p|dsdr\nonumber\\
	&\leq&\int_{B_R(y)\setminus B_\tau(y)}\frac{|\nabla p||\ln|x-y||^\ell}{|x-y|}dx\nonumber\\
	&\leq&\left(\int_{B_R(y)}|\nabla p|^2|\ln|x-y||^{3\ell} dx\right)^{\frac{1}{2}}\left(\int_{B_R(y)\setminus B_\tau(y)}\frac{1}{|\ln|x-y||^\ell|x-y|^{2}}dx\right)^{\frac{1}{2}}\nonumber\\
	&\leq&c(R,\ell)\left(	\int_{0}^{R}\frac{1}{|\ln  r|^\ell  r}d  r\right)^{\frac{1}{2}}\leq c(R,\ell).\nonumber
	%&=&c\left(\ln^2r-\ln^2R\right)^{\frac{1}{2}}\left(\int_{B_R(y)}|\nabla p|^2|\ln|x-y||^3 dx\right)^{\frac{1}{2}}.
\end{eqnarray}
Here we have employed \eqref{fln}. Observe that $\underset{ B_r(y)}{	\mbox{osc}}p_0$ is an increasing function of $r$. With this in mind, we obtain
\begin{eqnarray}
\underset{ B_\tau(y)}{	\mbox{osc}}	p_0&\leq& \frac{\int_{\tau}^{R}\frac{	|\ln r|^\ell}{r}\underset{ B_r(y)}{	\mbox{osc}}	p_0dr}{\int_{\tau}^{R}\frac{	|\ln r|^\ell}{r}dr}\nonumber\\
&\leq& \frac{c(R,\ell)}{\int_{0}^{\ln\frac{R}{\tau}}(s-\ln R)^\ell ds}\leq \frac{c(R,\ell)}{\ln^{\ell+1}\frac{R}{\tau}}.\nonumber
%\leq \frac{c\left(\ln^2r-\ln^2R\right)^{\frac{1}{2}}}{\ln^3R-\ln^3r}\left(\int_{B_R(y)}|\nabla p|^2|\ln|x-y||^3 dx\right)^{\frac{1}{2}}.|\ln\tau|^{\ell+1}-|\ln R|^{\ell+1}
\end{eqnarray}
Here we have used the fact that $\ln R<0$.
The above inequality together with \eqref{pb2} implies
%Consequently,
	%where $c$ is a positive number. +c\int_{r}^{R}\tau^{\frac{2(q-1)}{q}-1+\alpha}
%	Of course, the above inequality can also be established via an elementary calculus argument. Keeping the preceding estimates in mind,  we compute
	\begin{equation}
		\begin{split}
			\underset{ B_  \tau(y)}{	\mbox{osc}}p
			%=\sup_{ B_  tau(y)}p-\inf_{ B_  tau(y)}p\\
		%	=&\sup_{x_1, x_2\in B_r(y)}(p(x_1,t)-p(x_2,t))\\
		%	=&\sup_{x_1, x_2\in B_r(y)}(p_0(x_1,t)-p_0(x_2,t)+p_1(x_1,t)-p_1(x_2,t))\\
			\leq &\ \underset{ B_  \tau(y)}{	\mbox{osc}}p_0+\underset{ B_  \tau(y)}{	\mbox{osc}}p_1\\
			\leq & \frac{c(R,\ell)}{\ln^{\ell+1}\frac{R}{\tau}} +c  \tau^{\frac{2(q-1)}{q}}
			% \leq \frac{c}{\ln^{\ell+1}\frac{R}{\tau}}
			\ \ \mbox{for $\tau\in(0, R)$}.\nonumber
		\end{split}
	\end{equation}
	This finishes the proof of (C3).
\end{proof}

\begin{lem}\label{logb} Let $x_0\in\Omega$ be given and define
	$$ d(x_0)=\mbox{dist}(x_0,\ob).$$
	If $\m\in \left(\textup{BMO}(\Omega)\right)^N$, i.e.,
	$$\|\m\|_{\textup{BMO}(\Omega)}\equiv\sup_{\bRy\subset\Omega}\dashint|\m-\m_{y,r}|dx<\infty,$$
	then there exists a positive number $c=c(\|\m\|_{\textup{BMO}(\Omega)}, d(x_0))$ such that
	\begin{equation}
		|\m_{x_0,\rho}|\leq c\ln\left(\frac{ d(x_0)}{\rho}\right)+c\ \ \mbox{for each $\rho\in (0,  d(x_0))$.}\nonumber
	\end{equation}
\end{lem}
\begin{proof} We follow the proof of Proposition 1.2 in (\cite{G}, p. 68).
	For each $0<\rho\leq r\leq  d(x_0)$, we have
	$$|\m_{x_0,r}-\m_{x_0,\rho} |\leq|\m_{x_0,r}-\m(x)|+|\m(x)-\m_{x_0,\rho}|.$$
	Integrate the above inequality over $B_\rho(x_0)$ to derive
	\begin{equation}
		|\m_{x_0,r}-\m_{x_0,\rho} |\leq\left(\frac{r}{\rho}\right)^2\dashint_{B_r(x_0)}|\m_{x_0,r}-\m|dx+\dashint_{B_\rho(x_0)}|\m-\m_{x_0,\rho}|dx\leq c\left(\frac{r}{\rho}\right)^2.\nonumber
	\end{equation}
	In particular, take $r=\frac{d(x_0)}{2^i}, \rho=\frac{d(x_0)}{2^{i+1}}$, where $i\in \{0,1,\cdots\}$. Then we have
	\begin{equation}
		\left|\m_{x_0,\frac{d(x_0)}{2^i}}-\m_{x_0,\frac{d(x_0)}{2^{i+1}}}\right|\leq c.\nonumber
	\end{equation}
	For each $\rho\in (0, d(x_0))$ there is an $i\in \{0,1,\cdots\}$ such that
	\begin{equation}\label{j42}
		\frac{d(x_0)}{2^{i+1}}\leq \rho<\frac{d(x_0)}{2^i}.
	\end{equation}
	%$$$$\frac{1}{2}||\m_{x_0,r}|-|\m_{x_0,\rho}||\leq ||\m_{x_0,r}|-|\m_{x_0,\rho}||
	Subsequently,
	\begin{eqnarray}
		|\m_{x_0,\rho}|&\leq &\left|\m_{x_0,\frac{d(x_0)}{2^i}}\right|+\left|\m_{x_0,\rho}-\m_{x_0,\frac{d(x_0)}{2^i}}\right|\nonumber\\
		&\leq &\left|\m_{x_0,\frac{d(x_0)}{2^i}}-\m_{x_0,d(x_0)}\right|+\left|\m_{x_0,d(x_0)}\right|+c\nonumber\\
		&\leq &\sum_{j=1}^{i}\left|\m_{x_0,\frac{d(x_0)}{2^j}}-\m_{x_0,\frac{d(x_0)}{2^{j-1}}}\right|+\left|\m_{x_0,d(x_0)}\right|+c\nonumber\\
	%	&\leq &c\sum_{j=1}^{i}\left(\int_{B_{\frac{d(x_0)}{2^{j-1}}}(x_0)}|\nabla\m(x,t)|^{2}dx\right)^{\frac{1}{2}}+\left|\m_{x_0,d(x_0)}\right|\nonumber\\\left(\int_{B_{d(x_0)}(x_0)}|\nabla\m(x,t)|^{2}dx\right)^{\frac{1}{2}}
		&\leq &c(i+1)+\left|\m_{x_0,d(x_0)}\right|.\label{j43}
	\end{eqnarray}
	We easily see from \eqref{j42} that
	$$i\leq \frac{1}{\ln 2}\ln\left(\frac{d(x_0)}{\rho}\right).$$
	Using this in \eqref{j43} yields the desired result.
\end{proof}

We are in a position to prove (C4).
\begin{proof}[Proof of (C4)]
	%We can easily verify that
%	\begin{equation}\label{wi1}
%		\brh\subset B_r(y)\subset B_{\frac{3r}{2}}(x_0)\subset\Omega.
%	\end{equation}
	%	$$$$
%	Subsequently,
%	\begin{equation}\label{wi2}
	%	\ibrh  F\left|\ln|x-y|\right| dx\leq 	\ibRy  F\left|\ln|x-y|\right| dx.
%	\end{equation}
	Let $y\in\Omega, r>0$ be such that $B_{2r}(y)\subset\Omega$. Then for each $x_0\in\bRy$ and $\rho\in\left(0,\frac{r}{2}\right)$, we have $B_{2\rho}(x_0)\subset\Omega$. %Pick $\rho\in\left(0,\frac{r}{2}\right)$.
	Let $\rho$ be so chosen. Select a cutoff function $\zeta\in C_0^\infty(B_{2\rho}(x_0))$ the properties
$$\zeta=1\ \ \mbox{on $B_{\rho}(x_0)$, $0\leq \zeta\leq 1$ on $B_{2\rho}(x_0)$, and $|\nabla\zeta|\leq \frac{c}{\rho}$ on $B_{2\rho}(x_0)$}.$$
Use $(p-p_{x_0,2\rho})\zeta^2$ as a test function in \eqref{ell0} to obtain
\begin{eqnarray}
	\lefteqn{\int_{B_{2\rho}(x_0)}|\nabla p|^2\zeta^2dx+\int_{B_{2\rho}(x_0)}\mnp^2\zeta^2dx}\nonumber\\
	&=&-\int_{B_{2\rho}(x_0)}\nabla p\cdot\nabla\zeta\zeta(p-p_{x_0,2\rho})dx-\int_{B_{2\rho}(x_0)}\mnp\m\cdot\nabla\zeta\zeta(p-p_{x_0,2\rho})dx\nonumber\\
	&&+\int_{B_{2\rho}(x_0)}S(x)(p-p_{x_0,2\rho})\zeta^2dx\nonumber\\
	&\leq&\frac{1}{2}\int_{B_{2\rho}(x_0)}\left(|\nabla p|^2+\mnp^2\right)\zeta^2dx+\frac{c}{\rho^2}\int_{B_{2\rho}(x_0)}(p-p_{x_0,2\rho})^2dx\nonumber\\
	&&+\frac{c}{\rho^2}\int_{B_{2\rho}(x_0)}|\m|^2(p-p_{x_0,2\rho})^2dx+\int_{B_{2\rho}(x_0)}S(x)(p-p_{x_0,2\rho})\zeta^2dx,\label{xu2}\nonumber
\end{eqnarray}
from whence follows
\begin{eqnarray}
		\int_{B_{\rho}(x_0)}   F dx&=&\int_{B_{\rho}(x_0)}|\nabla p|^2dx+\int_{B_{\rho}(x_0)}\mnp^2dx\nonumber\\
	&\leq&c\left(\underset{B_{2\rho}(x_0)}{\textup{osc} }p\right)^2+\frac{c}{\rho^2}\int_{B_{2\rho}(x_0)}|\m-\m_{x_0,2\rho}|^2(p-p_{x_0,2\rho})^2dx\nonumber\\
	&&+c|\m_{x_0,2\rho}|^2\left(\underset{B_{2\rho}(x_0)}{\textup{osc}}p\right)^2+c\|S\|_{q,B_{\rho}(x_0)}\rho^{\frac{2(q-1)}{q}}\underset{B_{2\rho}(x_0)}{\textup{osc}}p\nonumber\\
	&\leq&c\left(1+(|\ln\rho|+1)^2\right)\left(\underset{B_{2\rho}(x_0)}{\textup{osc}}p\right)^2+c\rho^{\frac{2(q-1)}{q}}.
	%\frac{c}{\ln^{2\ell}\frac{R}{\rho}}+\frac{c(|\ln\rho|+1)^2}{\ln^{2\ell}\frac{R}{\rho}}+c\rho^{\frac{2(q-1)}{q}}
%	\ra 0\ \mbox{as $\rho \ra 0$}.
	\label{xu1}%\leq
	%\frac{c}{\ln^{2\ell-2}\frac{R}{\rho}}.
\end{eqnarray}
Here we have applied the John-Nirenberg Theorem \cite{DF}, which asserts that for each $s>1$ and subdomain $\Omega_0$ of $\Omega$ with $\overline{\Omega_0}\subset\Omega$ %. Then John-Nirenberg theorem \cite{DF} asserts that for each $s>1$
  there is a constant $c$ such that
 \begin{equation}\label{jn}
 	\sup_{\bRy\subset\Omega_0} \left(\dashint_{\bRy}|\m-\m_{y,r}|^sdx\right)^{\frac{1}{s}}\leq c\|\m\|_{\textup{BMO}(\Omega)}.
 \end{equation}
We derive from \eqref{wish1}, \eqref{xu1}, and (C3) that
\begin{eqnarray}
\lefteqn{	\int_{B_{\rho}(x_0)}F\chi_{\bRy}|\ln|x-x_0|dx}\nonumber\\
&\leq&\int_{B_{\rho}(x_0)}F|\ln|x-x_0|dx\nonumber\\
&\leq&(\ln2+|\ln \rho|)	\int_{B_{\rho}(x_0)}   F dx+c\int_{B_{\rho}(x_0)} |\nabla p|^2dx\nonumber\\
&&+c\left(\int_{B_\rho(x_0)}|\m|^4dx\right)^{\frac{1}{2}}\int_{B_{\rho}(x_0)} |\nabla p|^2dx+c\rho^{\frac{2(q-1)}{q}}\| S(x)\|_{q, B_\rho(x_0)}\nonumber\\
&\leq&(c+|\ln \rho|)	\int_{B_{\rho}(x_0)}   F dx+c\rho^{\frac{2(q-1)}{q}}\nonumber\\
&\leq&c\left(1+(|\ln\rho|+1)^2\right)(c+|\ln \rho|)\left(\underset{B_{2\rho}(x_0)}{\textup{osc}}p\right)^2+c(c+|\ln \rho|)\rho^{\frac{2(q-1)}{q}}\nonumber\\
&\leq&\frac{c\left(1+(|\ln\rho|+1)^2\right)(c+|\ln \rho|)}{\ln^{2\ell}\frac{R}{\rho}}+c(c+|\ln \rho|)\rho^{\frac{2(q-1)}{q}}\ \ 	\ra 0\ \mbox{as $\rho \ra 0$},
\end{eqnarray}
provided that $\ell>\frac{3}{2}$. 
% Combining \eqref{xu1} with \eqref{wish1} yields the desired result. 
The proof is complete.
\end{proof}
%Observe that Let $\Omega_0$ be a 
%\begin{eqnarray}
%\lefteqn{	\frac{1}{\rho^2}\int_{B_{2\rho}(x_0)}|\m|^2(p-p_{x_0,2\rho})^2dx}\nonumber\\
%&\leq&\frac{2}{\rho^2}\int_{B_{2\rho}(x_0)}|\m-\m_{x_0,2\rho}|^2(p-p_{x_0,2\rho})^2dx+\frac{2|\m_{x_0,2\rho}|^2}{\rho^2}\int_{B_{2\rho}(x_0)}(p-p_{x_0,2\rho})^2dx
%\end{eqnarray}
%We deduce from \eqref{xu2} and Poincar'{e}'s inequality on balls (\cite{EG}, p. 141) that
%\begin{eqnarray}
%	\int_{B_{\rho}(x_0)}|\nabla p|^2dx\leq c\rho^2\left(\dashint_{B_{2\rho}(x_0)}|\nabla p|dx\right)^2
%\end{eqnarray}
\section{Proof of Theorem \ref{part}}\label{s3}
In this section we prove the last two results in the introduction.
\begin{proof}[Proof of Theorem \ref{part}]
	For each $r\in (0,1)$ and $x_0\in\Omega$ such that
	$B_{2r}(x_0)\subset\Omega$, we define
	\begin{equation}\label{odef}
		\lp=\underset{\br}{\textup{ess sup}}\ p,\ \ \smp=\underset{\br}{\textup{ess inf}}\ p, \ \ \op=\underset{\br}{\textup{osc}}p=\lp-\smp.\nonumber
	\end{equation}
Let $q$ be given as in (H1).
	Then we can pick a number
	\begin{equation}\label{ddef}
		\delta\in \left(0,\frac{2q-2}{q}\right).
	\end{equation}
	Subsequently, set
	\begin{equation}
		\kbr=r^{\delta}\|S\|_{q,\br}.\nonumber
	\end{equation}
Either
	\begin{eqnarray}
		\left|	\br\cap\left\{p(x)\leq \smp+\frac{\op}{2}\right\}\right|&\geq&\frac{1}{2}|\br|,\ \ \mbox{or}\label{case1}\\
		\left|	\br\cap\left\{p(x)> \smp+\frac{\op}{2}\right\}\right|&\geq&\frac{1}{2}|\br|.\label{case2}
	\end{eqnarray}
	%Let $\delta$ be a positive number to be determined. 
	If \eqref{case1} is true, we consider the function
	$$w= \ln\left(\frac{ \op+\kbr }{2(\lp-p)+\kbr}\right). $$
	%where $$ and $\delta>0$ will be picked later.
	If \eqref{case2} holds, we take
	$$w= \ln\left(\frac{ \op+\kbr }{2(p-\smp)+\kbr}\right).$$
	%where $\kbr$ is the same as before.
	For definiteness, we assume the first case \eqref{case1}. Then we easily verify that $w$ satisfies the equation
	\begin{eqnarray}
		\lefteqn{-\divg\left[(I+\m\otimes\m)\nabla w\right]+(I+\m\otimes\m)\nabla w\cdot\nabla w}\nonumber\\
		&=&\frac{2 S}{2(\lp-p)+\kbr}\ \ \mbox{in $\br$.}\label{we1}
	\end{eqnarray}
	The rest of the proof is divided into several claims.
	\begin{clm}\label{enr}
		There is a $\delta_0\in (0,1)$ such that
		\begin{eqnarray}
			%	\dashint_{B_\rho(x_0)}\lefteqn{}\nonumber\\
				\dashint_{B_{(1-\delta_0)r}(x_0)}\left(w^+\right)^2dx
			&\leq& c\aibr|\m|^2dx+cr^{2-\frac{2}{q}-\delta}+c.\nonumber% \ \ \mbox{for each $r\in(0,T)$}.
			%	\limsup_{\rho\ra 0}\dashint_{B_\rho(x_0)} \left(w^+\right)^2dx<\infty.\frac{c}{(1-\delta_0)^2\delta_0^2}\frac{c}{(1-\delta_0)^2}\frac{c}{(1-\delta_0)^2\delta_0^2}
		\end{eqnarray}
	\end{clm}
	\begin{proof} Select $\delta_0\in(0,1)$ as below. 	Let $\theta(\eta)$ be a smooth decreasing function on $[0,r]$ with the properties
		$$\theta(\eta)=1 \ \ \mbox{for $\eta\leq (1-\delta_0)r$, $\theta(r)=0$, and $|\theta^\prime(\eta)|\leq \frac{c}{r\delta_0}$ on $[0,r]$.}$$
		We easily see that
		$$\left|\br\setminus B_{(1-\delta_0)r}(x_0)\right|\leq N\delta_0|\br|.$$
		We can pick $\delta_0$ so that
		\begin{equation}\label{ho1}
			|\{\theta(|x-x_0|)=1\}\cap\{x\in B_{r}, w^+(x,t)=0\}|\geq c_0|\br| \ \ \mbox{some $c_0>0$}.	
			%\ \ \mbox{for each $t\in\itz$}	
		\end{equation}
	%	It is enough for us to 
		To see this, take
		\begin{equation}\label{hope15}
			N\delta_0=\frac{1}{16}. \nonumber
		\end{equation}
	%	Indeed, we can easily check that
	Note that
		\begin{equation}
			w^{+}(x,t)=0 \Leftrightarrow x\in\br\cap\left\{w(x,t)\leq \smp+\frac{\op}{2}\right\}.\nonumber
		\end{equation}
		Combining this with \eqref{case1}  yields \eqref{ho1} with $c_0=\frac{7}{16}$.
		Now we are in a position to apply Proposition 2.1 in (\cite{D}, p.5). Upon doing so, we obtain
		\begin{equation}\label{sob2}
			\ibr \theta^2(|x-x_0|)(w^+)^2dx\leq cr^2\ibr \theta^2(|x-x_0|)|\nabla w|^2dx.
		\end{equation}
%	We can easily check that
%	\begin{eqnarray}
%		\m\otimes\m&=&\m_{x_0,r}\otimes\m_{x_0,r}+(\m-\m_{x_0,r})\otimes(\m-\m_{x_0,r})+\nonumber\\
%		&&+(\m-\m_{x_0,r})\otimes \m_{x_0,r}+\m_{x_0,r}\otimes(\m-\m_{x_0,r}),\\
%		\nabla w&=&\frac{2}{2(\lp-p)+\kbr}\nabla p.
%	\end{eqnarray}
	%Plug this into \eqref{ell0} and	u
	Use $\theta^2$ as a test function in \eqref{we1} to derive
\begin{eqnarray}
\lefteqn{\ibr\left(|\nabla w|^2+\mnw^2\right)\theta^2dx}\nonumber\\
&	=&-2\ibr\nabla w\cdot\nabla\theta\theta dx-2\ibr\mnw\m\cdot\nabla\theta\theta dx\nonumber\\
&&+\ibr\frac{2S\theta^2}{2(\lp-p)+\kbr}dx.\label{sob3}
\end{eqnarray}
We easily see that
\begin{equation}\label{sob4}
\frac{1}{2(\lp-p)+r^\delta\|S\|_{q,\br}}\leq \frac{1}{r^\delta\|S\|_{q,\br}}.\nonumber
\end{equation}
%	from whence follows
We shall always use this inequality when dealing with the last term in \eqref{we1}. %Therefore, our generic constant $c$ does not depend on $a$ in our subsequent calculations until we make the selection for $a$ at the end. 
Indeed, with this in mind, we can derive from \eqref{sob3} that
\begin{equation}
\ibr|\nabla w|^2\theta^2dx\leq c+\frac{c}{ r^2}\ibr|\m|^2dx+cr^{2-\frac{2}{q}-\delta}.\nonumber
\end{equation}
This together with \eqref{sob2} yields
\begin{eqnarray}
\int_{B_{(1-\delta_0)r}(x_0)} (w^+)^2 dx&\leq &	\ibr\theta^2(w^+)^2 dx\nonumber\\
&\leq& cr^2\ibr\theta^2|\nabla w|^2dx\nonumber\\
&\leq& cr^2+cr^2\aibr|\m|^2dx+cr^{4-\frac{2}{q}-\delta}.\nonumber
\end{eqnarray}
The lemma follows.
	
	\end{proof}
	\begin{clm}\label{wub}We have
		\begin{eqnarray}
			\sup_{\brh} w&\leq&c\left(\dashint_{B_r(x_0)}\left(w^+\right)^2dx\right)^{\frac{1}{2}}+\frac{1}{r}\|\m\|_{4,\br}^2+r^{2-\delta-\frac{2}{q}}.\nonumber
			%c\left(\dashint_{B_r(x_0)}\left(w^+\right)^2dx\right)^{\frac{1}{2}}+r^{2-N}\|\m\|_{\frac{2N}{N-2},\br}^2+r^{2-\delta-\frac{N}{s}}.
		\end{eqnarray}
	\end{clm}
	\begin{proof} We employ the classical De Giorgi iteration scheme. For this purpose, we define
		%Set due to Moser-Nash-
		$$r_n=\frac{r}{2}+\frac{r}{2^{1+n}}, \ \ n=0,1,\cdots.$$
		Select a sequence of smooth functions $\zeta_n(x)$ so that
		$$\zeta_n(x)=1\ \ \mbox{on $\brn$, $\zeta_n(x)=0$ outside $B_{r_{n-1}}(x_0)$, $|\nabla\zeta_n|\leq \frac{c2^n}{r}$, and $0\leq\zeta_n\leq 1$.}$$
		Choose 
		%	\begin{equation}\label{kcon2}
			%		k\geq r^{2-N}\|\m\|_{\frac{2N}{N-2},\br}^2+r^{2-\delta-\frac{N}{q}}
			%	\end{equation}
		$k>0$ as below. Set
		$$k_n=k-\frac{k}{2^{n+1}}\ \ \mbox{for $n=0,1,\cdots$.} $$
		Use $(w-k_{n+1})^+\zeta_{n+1}^2$ as a test function in \eqref{we1} and make use of the fact that
		$$\nabla w= \nabla(w-k_{n+1})^+\ \ \mbox{on $\{w\geq k_{n+1}\}$}$$
		to derive
		\begin{eqnarray}
			\lefteqn{\ibnx\left(|\nabla(w-k_{n+1})^+|^2+\left[\m\cdot\nabla(w-k_{n+1})^+\right]^2\right)\zeta_{n+1}^2dx}\nonumber\\
			&&+\ibnx(I+\m\otimes\m)\nabla(w-k_{n+1})^+\cdot\nabla(w-k_{n+1})^+(w-k_{n+1})^+\zeta_{n+1}^2dx\nonumber\\
			&=&-2\ibnx\nabla (w-k_{n+1})^+\cdot\nabla\zeta_{n+1}\zeta_{n+1}(w-k_{n+1})^+dx\nonumber\\
			&&-2\ibnx\m\cdot\nabla(w-k_{n+1})^+\m\cdot\nabla\zeta_{n+1}\zeta_{n+1}(w-k_{n+1})^+dx\nonumber\\
			&&+\ibnx\frac{2S(w-k_{n+1})^+\zeta_{n+1}^2}{2(\lp-p)+r^\delta\|S\|_{q,\br}}dx.\label{we2}
		\end{eqnarray}
		Fix  $s\geq\max\{2, \frac{q}{q-1}\}$ so that 
		$$ \frac{s}{s-1}\leq q\ \ \mbox{and }\ \ \frac{2s}{s+2}\geq 1.$$
		We can use the Sobolev inequality to estimate the last integral in \eqref{we2} as follows:
		\begin{eqnarray}
			\lefteqn{	\ibnx\frac{2S(w-k_{n+1})^+\zeta_{n+1}^2}{2(\lp-p)+r^\delta\|S\|_{q,\br}}dx}\nonumber\\
			&\leq&\frac{2}{r^\delta\|S\|_{q,\br}}\|S\|_{\frac{s}{s-1},\brn\cap\{w\geq k_{n+1}\}}\|(w-k_{n+1})^+\zeta_{n+1}\|_{s,\brn}\nonumber\\
			&\leq&\frac{2}{r^\delta}\left|\brn\cap\{w\geq k_{n+1}\}\right|^{\frac{s-1}{s}-\frac{1}{q}}\|\nabla\left[(w-k_{n+1})^+\zeta_{n+1}\right]\|_{\frac{2s}{s+2},\brn}\nonumber\\
				&\leq&\frac{2}{r^\delta}\left|\brn\cap\{w\geq k_{n+1}\}\right|^{1-\frac{1}{q}}\|\nabla\left[(w-k_{n+1})^+\zeta_{n+1}\right]\|_{2,\brn}\nonumber\\
			&\leq&\ve \|\nabla\left[(w-k_{n+1})^+\zeta_{n+1}\right]\|_{2,\brn}^2+\frac{c}{\ve}\left|\brn\cap\{w\geq k_{n+1}\}\right|^{2-\delta-\frac{2}{q}},\ \ \ve>0.\nonumber
		\end{eqnarray}
		Here we have used the fact that $\left|\brn\cap\{w\geq k_{n+1}\}\right|\leq c r^2$.
		Substitute this into \eqref{we2}, choose $\ve$ suitably small in the resulting equation, and thereby obtain
		%from whence follows
		\begin{eqnarray}
			\lefteqn{\ibnx|\nabla\left[(w-k_{n+1})^+\zeta_{n+1}\right]|^2dx}\nonumber\\
			&\leq&\frac{c4^n}{r}y_n+\frac{c4^n}{r^2}\ibnx|\m|^2\left[(w-k_{n+1})^+\right]^2dx\nonumber\\
			&&+c\left|\brn\cap\{w\geq k_{n+1}\}\right|^{2-\delta-\frac{2}{q}},\nonumber
		\end{eqnarray}
		where
		%Introduce the sequence
		$$y_n=\left(\ibnx\left[(w-k_{n})^+\right]^4dx\right)^{\frac{1}{2}}.$$
		It is easy to see that
		\begin{eqnarray}
			y_n\geq \left(\int_{B_{r_n}(x_0)\cap\{w>k_{n+1}\}}\left[(w-k_{n})^+\right]^4dx\right)^{\frac{1}{2}}\geq \frac{k^2}{4^{n+2}}\left|B_{r_n}(x_0)\cap\{w>k_{n+1}\}\right|^{\frac{1}{2}}.\label{j41}
		\end{eqnarray}
		%	$$$$
		Let $s$ be given as before. 	We derive from the Sobolev inequality that
		\begin{eqnarray}
			y_{n+1}&\leq&\ibnx\left[(w-k_{n+1})^+\zeta_{n+1}\right]^2dx\nonumber\\
			&\leq&\left(\ibnx\left[(w-k_{n+1})^+\zeta_{n+1}\right]^{2s}dx\right)^{\frac{1}{s}}\left|B_{r_n}(x_0)\cap\{w>k_{n+1}\}\right|^{\frac{s-1}{s}}\nonumber\\
			&\leq&c\left(\ibnx\left|\nabla\left[(w-k_{n+1})^+\zeta_{n+1}\right]\right|^{\frac{2s}{1+s}}dx\right)^{\frac{1+s}{s}}\left|B_{r_n}(x_0)\cap\{w>k_{n+1}\}\right|^{\frac{s-1}{s}}\nonumber\\
			&\leq&c\ibnx\left|\nabla\left[(w-k_{n+1})^+\zeta_{n+1}\right]\right|^{2}dx\left|B_{r_n}(x_0)\cap\{w>k_{n+1}\}\right|\nonumber\\
			&\leq&\frac{c4^n}{r^2}y_n\left|B_{r_n}(x_0)\cap\{w>k_{n+1}\}\right|\nonumber\\
			&&+\frac{c4^n}{r^2}\ibnx|\m|^2\left[(w-k_{n+1})^+\right]^2dx\left|B_{r_n}(x_0)\cap\{w>k_{n+1}\}\right|\nonumber\\
			&&+c\left|\brn\cap\{w\geq k_{n+1}\}\right|^{3-\delta-\frac{2}{q}}.\label{we3}
			%\right)\left(\frac{4^{n+2}y_n}{k^2}\right)^{\frac{2}{N}}
		\end{eqnarray}
		Set
		$$\sigma=\min\left\{2-\delta-\frac{2}{q}, \ \frac{1}{2}\right\}>0\ \ \mbox{due to \eqref{ddef}.}$$
		Then we obtain from \eqref{j41} that
		\begin{eqnarray}
			y_n\left|B_{r_n}(x_0)\cap\{w>k_{n+1}\}\right|&=&y_n\left|B_{r_n}(x_0)\cap\{w>k_{n+1}\}\right|^{1-\sigma+\sigma}\nonumber\\
			&\leq &cr^{2(1-\sigma)}\left(\frac{4^{n+2}y_n}{k^2}\right)^\sigma y_n\leq \frac{c4^{\sigma n}r^{2}}{r^{2\sigma}k^{2\sigma}}y_n^{1+\sigma}.\nonumber
		\end{eqnarray}
		%	We calculate from \eqref{kcon2} that
		Similarly,
		\begin{eqnarray}
			\lefteqn{\ibnx|\m|^2\left[(w-k_{n+1})^+\right]^2dx\left|B_{r_n}(x_0)\cap\{w>k_{n+1}\}\right|}\nonumber\\
			&\leq&\|\m\|_{4,\br}^2\left(\ibnx\left[(w-k_{n+1})^+\right]^4dx\right)^{\frac{1}{2}}\left|B_{r_n}(x_0)\cap\{w>k_{n+1}\}\right|^{\frac{1}{2}+\sigma+\frac{1}{2}-\sigma}\nonumber\\
		%	&\leq&\|\m\|_{4,\br}^2y_n^{\frac{1}{2}}\left|B_{r_n}(x_0)\cap\{w>k_{n+1}\}\right|^{\frac{1}{2}+\sigma+\frac{1}{2}-\sigma}\nonumber\\
			%	&\leq&\|\m\|_{4,\br}^2y_n^{\frac{1}{2}}\left|B_{r_n}(x_0)\cap\{w>k_{n+1}\}\right|^{\frac{4}{N}-\frac{1}{2}}\nonumber\\
			&\leq&\underset{\bRy}{\textup{ess sup}}\ w^+\|\m\|_{4,\br}^2y_n^{\frac{1}{2}}r^{1-2\sigma }\left(\frac{4^{n+2}y_n}{k^2}\right)^{\frac{1}{2}+\sigma}\nonumber\\
			&\leq& \frac{c4^{\left(\frac{1}{2}+\sigma\right) n}r^{2}}{r^{2\sigma}k^{2\sigma}}y_n^{1+\sigma}.\nonumber
		\end{eqnarray}
		Here we have assumed that
		\begin{eqnarray}
			k\geq \frac{1}{r}\|\m\|_{4,\br}^2.\label{kcon2}
		\end{eqnarray}
		%	Similarly, use \eqref{kcon2} to estimate 
		As for the last term in \eqref{we3}, we have% to get % can be estimated as follows:
		\begin{eqnarray}
			\lefteqn{\left|\brn\cap\{w\geq k_{n+1}\}\right|^{1+\sigma+2-\delta-\frac{2}{q}-\sigma}}\nonumber\\
			&\leq&cr^{4-2\delta-\frac{4}{q}-2\sigma }\left(\frac{4^{n+2}y_n}{k^2}\right)^{1+\sigma}\leq \frac{c4^{\left(1+\sigma\right) n}}{r^{2\sigma}k^{2\sigma}}y_n^{1+\sigma}.\nonumber
		\end{eqnarray}
		In this case, we have taken
		\begin{equation}\label{k3}
			k\geq r^{2-\delta-\frac{2}{q} }.
		\end{equation}
		Collecting these estimates in \eqref{we3} yields
		\begin{equation}
			y_{n+1}\leq \frac{c4^{\left(\frac{3}{2}+\sigma\right) n}}{r^{2\sigma}k^{2\sigma}}y_n^{1+\sigma}.\nonumber
		\end{equation}
		According to Lemma 4.1 in (\cite{D}, p.12), if we pick $k$ so large that
		$$y_0\leq \frac{r^2k^2}{c^{\frac{1}{\sigma}}4^{\frac{\frac{3}{2}+\sigma}{\sigma^2}}},$$
		then
		$$\sup_{\brh}w\leq k.$$
		Recall that
		$$y_0=\ibr\left[\left(w-\frac{k}{2}\right)^+\right]^2dx\leq \ibr (w^+)^2dx.$$
		This together with \eqref{kcon2} and \eqref{k3} implies that it is enough for us to take
		$$k=c\left(\dashint_{B_r(x_0)}\left(w^+\right)^2dx\right)^{\frac{1}{2}}+\frac{1}{r}\|\m\|_{4,\br}^2+r^{2-\delta-\frac{2}{q}}.$$
		This gives the desired result.
	\end{proof}
	
	We are ready to complete the proof of Theorem \ref{part}.
%With this in mind, w
We estimate from \eqref{jn} that
	\begin{eqnarray}
		\frac{1}{r}\|\m\|_{4,\br}^2&=&c\left(\aibr|\m|^4dx\right)^{\frac{1}{2}}\nonumber\\
		&\leq& c\left(\aibr|\m-\m_{x_0,r}|^4dx\right)^{\frac{1}{2}}+c|\m_{x_0,r}|^2\nonumber\\
	%	&\leq&cr^2\left(\aibr|\nabla\m|^{\frac{4}{3}}dx\right)^{\frac{3}{2}}+c|\m_{x_0,r}|^2\nonumber\\c\ibr|\nabla\m|^{2}dx+c|\m_{x_0,r}|^2
		&\leq&\leq c|\m_{x_0,r}|^2+c.\label{re}\nonumber
	\end{eqnarray}
	Similarly, 
	$$\aibr|\m|^2dx\leq c|\m_{x_0,r}|^2+c.$$
	We conclude from Claims \ref{enr} and \ref{wub} that
	\begin{eqnarray}
		\sup_{B_{\frac{(1-\delta_0)r}{2}}(x_0)}w&\leq& c\left(\dashint_{B_{(1-\delta_0)r}(x_0)}w^+dx\right)^{\frac{1}{2}}+\frac{1}{r}\|\m\|^2_{4,\br}+cr^{2-\delta-\frac{2}{q}}\nonumber\\
		&\leq &c\left(\aibr|\m|^2dx\right)^{\frac{1}{2}}+cr^{1-\frac{1}{q}-\frac{\delta}{2}}+c+\frac{1}{r}\|\m\|^2_{4,\br}+cr^{2-\delta-\frac{2}{q}}\nonumber\\
		&\leq&\leq c|\m_{x_0,r}|^2+c\leq c.\nonumber
	\end{eqnarray}
The last step is due to \eqref{mav}.	By the definition of $w$,
	\begin{eqnarray}
		 \op+\kbr \leq e^{c}\left[2\left(\lp-p\right)+\kbr\right]\ \ \mbox{on $B_{\frac{(1-\delta_0)r}{2}}(x_0)$},\nonumber
	\end{eqnarray}
	from whence follows
	\begin{eqnarray}
		2e^{ c}M_{x_0, \frac{(1-\delta_0)r}{2}} &\leq&\left(	2e^{ c}-1\right)\op+	2e^{ c}\smp\nonumber\\
		&&+\left(e^{ c}-1\right)\kbr.\nonumber
	\end{eqnarray}
	Subtract $2e^{ c}m_{x_0, \frac{(1-\delta_0)r}{2}} $ from both sides of the above inequality and make use of the fact that $\smp$ is a decreasing function of $r$ to deduce.
	\begin{eqnarray}
		\omega_{x_0, \frac{(1-\delta_0)r}{2}} &\leq &\frac{2e^{ c}-1}{2e^{ c}}\op+\frac{e^{ c}-1}{2e^{ c}}\kbr.\nonumber
	%	&\leq &\frac{e^{ c}-1}{2e^{ c}}\op+\kbr.\label{hope14}\nonumber	
	\end{eqnarray}
	Obviously,  recall the definition of $\kbr$ to get
	$$\omega_{x_0, \frac{(1-\delta_0)r}{2}} \leq \frac{2e^c-1}{2e^c}\op+r^\delta\|S\|_{q,\br}\leq  \frac{2e^c-1}{2e^c}\op+cr^\delta.$$
	%Note that $\omega_r$ is an increasing function of $r$. We can conclude from \eqref{po4}, \eqref{hope14}, and our assumptions on $\qi$ that there exist $\delta\in(0,1)$ and $c>0$ such that
	%$$\omega_{\frac{r}{2}}\leq \delta\omega_r+cr^{2-\frac{N}{s}}\ \ \mbox{for each $r>0$ with $Q_{2r}(z_0)\subset \ot$.}$$
	It is very important to note that the two coefficients on the right-hand side of the above inequality can be made independent of $r$ for $ r$ small. This enables us to invoke Lemma 8.23 in (\cite{GT}, p.201) to get
	\begin{equation}\label{hope16}
		\omega_{x_0,\rho} \leq c\rho^\alpha \ \ \mbox{for some $c>0, \alpha>0$ and $\rho$ sufficiently small.}\nonumber
	\end{equation}
	The proof is complete.
\end{proof}

Finally, we give the proof of Lemma \ref{kato3}.
\begin{proof}[Proof of Lemma \ref{kato3}] We can offer a much simpler proof than the one in \cite{FGL} due to our boundary condition.
	Take
$$0<r\leq\frac{1}{2}$$
so that
%Then we have
\begin{equation}
	\ln |x-y|\leq 0\ \ \mbox{whenever $x, y\in\br$.}\nonumber
\end{equation}
Consider the problem
\begin{eqnarray}
	-\Delta \psi&=&F\ \ \mbox{in $\br$},\label{pp1}\\
	\psi&=&0\ \ \mbox{on $\partial\br$.}\label{pp2}
	%	\frac{\partial\psi}{\partial\nu}&=&\frac{1}{2\pi r}\ibr F dx\equiv f(r).
\end{eqnarray}
\begin{clm}
	There exists a constant $c$ such that
	\begin{equation}\label{pub}
		\sup_{\br}|\psi|\leq c\eta(F;\br;r),
	\end{equation}
	where $\eta(F;\br;r)$ is given as in \eqref{kato1}.
	\end{clm}
\begin{proof} We can write
	$$F=F^+-F^-.$$
	Then solve the following two problems
	\begin{eqnarray}
		-\Delta \psi_1&=&F^+\ \ \mbox{in $\br$},\label{ppp}\\
		\psi_1&=&0\ \ \mbox{on $\partial\br$}\nonumber
		%	\frac{\partial\psi}{\partial\nu}&=&\frac{1}{2\pi r}\ibr F dx\equiv f(r).
	\end{eqnarray}
	and
	\begin{eqnarray}
		-\Delta \psi_2&=&F^-\ \ \mbox{in $\br$},\nonumber\\
		\psi_2&=&0\ \ \mbox{on $\partial\br$.}\nonumber
		%	\frac{\partial\psi}{\partial\nu}&=&\frac{1}{2\pi r}\ibr F dx\equiv f(r).
	\end{eqnarray}
	Clearly,
	$$\psi=\psi_1-\psi_2.$$
	The non-negative function
	$$-\frac{1}{2\pi}\ibr F^+\ln|x-y|dy$$
	satisfies the equation \eqref{ppp}. The comparison principle asserts
	$$0\leq \psi_1\leq -\frac{1}{2\pi}\ibr F^+\ln|x-y|dy\ \ \mbox{on $\br$.}$$
	Similarly,
	$$0\leq \psi_2\leq -\frac{1}{2\pi}\ibr F^-\ln|x-y|dy\ \ \mbox{on $\br$.}$$
	The preceding inequalities yield the claim.
\end{proof}

The rest of the proof here largely follows \cite{FGL}. Without loss of generality, assume
	$$F\geq 0.$$
	Let $\psi$ be given as in \eqref{pp1}-\eqref{pp2}. For $u\in W^{1,2}_0(\br)\cap L^\infty(\br)$, we calculate
	\begin{eqnarray}
		\ibr Fu^2dx&=&-\ibr\Delta\psi u^2dx\nonumber\\
		&=&2\ibr u\nabla\psi\cdot\nabla udx\nonumber\\
		&\leq&2\left(\ibr|\nabla u|^2dx\right)^{\frac{1}{2}}\left(\ibr u^2|\nabla \psi|^2dx\right)^{\frac{1}{2}}.\label{pp3}
	\end{eqnarray}
	It is easy to verify
	$$|\nabla\psi|^2=\frac{1}{2}\Delta\psi^2-\psi\Delta\psi=\frac{1}{2}\Delta\psi^2-\psi F.$$
	This together with \eqref{pub} implies
	\begin{eqnarray}
		\ibr u^2|\nabla \psi|^2dx&=&\frac{1}{2}\ibr u^2\Delta\psi^2dx-\ibr\psi Fu^2dx\nonumber\\
		&\leq&-	\frac{1}{2}\ibr \nabla u^2\cdot\nabla\psi^2dx+\|\psi\|_{\infty,\br}\ibr Fu^2dx\nonumber\\
		&\leq&\frac{1}{2}\ibr u^2|\nabla \psi|^2dx+2\ibr \psi^2|\nabla u|^2dx\nonumber\\
		&&+c\eta(F;\br;r)\ibr Fu^2dx,\nonumber
	\end{eqnarray}
	from whence follows
	\begin{equation}
		\ibr u^2|\nabla \psi|^2dx\leq c\eta^2(r)\ibr|\nabla u|^2dx+c\eta(F;\br;r)\ibr Fu^2dx.	\nonumber
	\end{equation}
	Plug this into \eqref{pp3} to derive
	\begin{eqnarray}
	\ibr|F|u^2dx
	&\leq& c\eta(F;\br;r)\ibr|\nabla u|^2dx\nonumber\\
	&&+c\left(\eta(F;\br;r)\ibr Fu^2dx\right)^{\frac{1}{2}}\left(\ibr|\nabla u|^2dx\right)^{\frac{1}{2}}\nonumber\\
		&\leq& \frac{1}{2}\ibr|F|u^2dx+c\eta(F;\br;r)\ibr|\nabla u|^2dx.\nonumber
	\end{eqnarray}
	The lemma follows.
\end{proof}

\end{document}